\documentclass[]{amsart} \usepackage{amsfonts}
\usepackage{bbding}
 \usepackage[dvips]{graphicx}
\usepackage{amsfonts,amssymb,amsmath,amsthm,amscd}

\usepackage{mathrsfs}
\usepackage{xcolor}
\usepackage{empheq}
\usepackage{adforn}
\usepackage{cancel}
\usepackage{xr}
\usepackage{mdframed}
\usepackage{tikz}
\usepackage{amsfonts}
\usepackage{mathdots}
\usepackage{enumerate}
\usepackage{lmodern}
\usepackage{mdframed}
\usepackage{stmaryrd}
\usepackage{blindtext}
\usepackage[all, knot]{xy}
\usepackage{leftidx}
\usepackage{accents}

\usepackage[colorlinks=true, pdfstartview=FitP, linkcolor=blue!80!black, citecolor=red!80!black, urlcolor=blue!80!black]{hyperref}

\usepackage{tgpagella}
\usepackage[T1]{fontenc}

\newtheorem*{theorem*}{Theorem}

\newtheorem*{lemma*}{Lemma}
 
\newtheorem*{corollary*}{Corollary}

\theoremstyle{definition}

\theoremstyle{remark}  \numberwithin{equation}{section}
\numberwithin{figure}{section}

\newcommand{\area}{\mathsf{Area}}

\newcommand{\vect}[1]{{\boldsymbol{#1}}}

\newcommand{\dif}{\mathsf{d}}

\newcommand{\ac}[1]{{\boldsymbol{#1}}} 

\newcommand{\dz}{{\dif z}}

\newcommand{\dx}{{\dif x}}

\newcommand{\PSL}{\mathrm{PSL}}

\newcommand{\eg}{\textit{e.g. }}

\newcommand{\ie}{\textit{i.e. }}

\newcommand{\volen}{\mathsf{Ent}}

\newcommand{\V}{V}
\renewcommand{\H}{\mathsf{H}}
\newcommand{\B}{\mathsf{B}}

\newcommand{\hyp}{\mathsf{hyp}}

\begin{document}

\title{Entropy degeneration of convex projective surfaces} 

\author[Xin Nie]{Xin Nie} 
\address{School of Mathematics, KIAS, 85 Hoegiro, Dongdaemun-gu, Seoul 130-722, Republic of Korea.}
\email{nie.hsin@gmail.com}
\thanks{The research leading to these results has received funding from the European Research Council under the {\em European Community}'s seventh Framework Programme (FP7/2007-2013)/ERC {\em grant agreement} ${\rm n^o}$ FP7-246918}

\maketitle


\begin{abstract}
We show that the volume entropy of the Hilbert metric on a closed convex projective surface tends to zero as the corresponding Pick differential tends to infinity. 
The proof is based on the fact, due to Benoist and Hulin, that the Hilbert metric and the Blaschke metric are comparable.
\end{abstract}

\section{Introduction and statement of results}

Let $\Sigma$ be a closed oriented surface of genus at least $2$. A \emph{strictly convex real \mbox{projective} structure} (referred to simply as ``convex projective structure'' in the sequel) is by definition a $(\PSL(3,\mathbb{R}),\mathbb{RP}^2)$ geometric structure whose developing map is a homeomorphism from the universal cover $\widetilde{\Sigma}$ to a bounded convex open set in an affine chart $\mathbb{R}^2\subset\mathbb{RP}^2$.

It is a theorem of J. Loftin  \cite{loftin_amer} and F. Labourie \cite{labourie_cubic} that the moduli space $\mathcal{P}(\Sigma)$ of convex projective structures on $\Sigma$ naturally identifies with the moduli space $\mathcal{C}(\Sigma)$ of pairs $(\ac{J},\vect{b})$, where $\ac{J}$ is a conformal structure on $\Sigma$ and $\vect{b}$ is a holomorphic cubic differential on the Riemann surface $(\Sigma,\ac{J})$. The space $\mathcal{C}(\Sigma)$ is a holomorphic vector bundle of rank $5(g-1)$ over the Teichm\"uller space $\mathcal{T}(\Sigma)$, the fiber over $(\ac{J},\vect{b})$ being the space $H^0((\Sigma,\ac{J}),K^3)$ of holomorphic cubic differentials on $(\Sigma,\vect{J})$.  The natural inclusion $\mathcal{T}(\Sigma)\subset\mathcal{P}(\Sigma)\cong\mathcal{C}(\Sigma)$ identifies $\mathcal{T}(\Sigma)$ as the zero section of the vector bundle.

Taking the volume entropy of the Hilbert metric (see Section \ref{sec_proof} below for the definitions) for each element of $\mathcal{P}(\Sigma)$ yields a function $\volen:\mathcal{P}(\Sigma)\rightarrow\mathbb{R}_+$.  

The Hilbert metric for a point in the Teichm\"uller space $\mathcal{T}(\Sigma)\subset\mathcal{P}(\Sigma)$ is just the hyperbolic metric representing that point. Its entropy is just $1$. On the other hand, M. Crampon proved in \cite{crampon} that $\volen$ is strictly less than $1$ outside $\mathcal{T}(\Sigma)$. 
Motivated by a question of Crampon, we constructed in \cite{nie_1} certain paths in $\mathcal{P}(\Sigma)$ along which $\volen$ tends to $0$ (the construction works for dimensions $3$ and $4$ as well, here we only look at dimension $2$ though). Using different techniques, T. Zhang constructed in  \cite{zhang} some submanifolds of $\mathcal{P}(\Sigma)$ with the same property. In view of these constructions, one naturally guess that  $\volen$ tends to $0$ along any sequence going away from $\mathcal{T}(\Sigma)$.
The purpose of this note is to prove such a statement. Namely,
\begin{theorem*}
Fix a conformal structure $\ac{J}$ on $\Sigma$. Let $(\vect{b}_n)$ be a sequence in the space of cubic differentials $V:=H^0((\Sigma,\ac{J}), K^3)$ and let $\delta_n$ be the volume entropy of Hilbert metric of the convex projective structure corresponding to $(\ac{J},\vect{b}_n)$. Then $\delta_n$ tends to $0$ as $n\rightarrow+\infty$ if and only if $\vect{b}_n$ tends to infinity in $V$.
\end{theorem*}
The ``only if'' part is an immediate consequence of the continuity of the map $\volen:\mathcal{P}(\Sigma)\rightarrow\mathbb{R}_+$ and the continuity of the Labourie-Loftin bijection $\mathcal{P}(\Sigma)\cong\mathcal{C}(\Sigma)$. However, to the knowledge of the author, a proof of the latter does not exist in the literature, so we will give an alternative proof of the above theorem without using these continuities.

The theorem is a simple manifestation of the general philosophy that a degenerating convex projective surface looks bigger and bigger and more and more flat. See \cite{loftin_limit, parreau} for more circumstantial statements confirming this philosophy.

\section{The proof and speculations}\label{sec_proof}
We first briefly review the backgrounds.

Let $\widetilde{\Sigma}$ be the universal cover of $\Sigma$ and $\Gamma=\pi_1(\Sigma)$ be the fundamental group. Let  $g$ be either a Riemannian metric, a Finsler metric or a flat metric with conic singularities on $\Sigma$.
Fix a base point $x_0\in\widetilde{\Sigma}$.
The volume entropy of  $g$ can be defined as
$$
\volen(g)=\limsup_{R\rightarrow+\infty}\frac{1}{R}\log \#(B_g(x_0,R)\cap \Gamma.x_0)\in\mathbb{R}_{\geq 0}\cup \{+\infty\},
$$
where the symbol ``$\#$'' means taking the cardinal of a set,  while $\Gamma.x_0$ is the $\Gamma$-orbit of $x_0$ and $B_g(x_0,R)\subset\widetilde{\Sigma}$ is the ball of radius $R$ centered at $x_0$ with respect to the distance on $\widetilde{\Sigma}$ induced by the lift of $g$. 

We will make use of the following properties of the entropy:
\begin{itemize}
\item  The scaling property: $\volen(tg)=t^{-1}\volen(g)$ for any $t>0$.
\item Given metrics $g_1$ and $g_2$, if their induced distances $d_1$ and $d_2$ on $\widetilde{\Sigma}$ are quasi-isometric in the sense that
\begin{equation}\label{eqn_quasiiso}
a^{-1}d_1(\,\cdot\,,\,\cdot\,)-b\leq d_2(\,\cdot\,,\,\cdot\,)\leq  a\,d_1(\,\cdot\,,\,\cdot\,)+b,
\end{equation}
for some constants $a>1$ and $b>0$, then $$a^{-1}\volen(g_1) \leq\volen(g_2)\leq a\,\volen(g_1).$$
\item A theorem of Katok \cite{katok}: for any Riemannian metric $g$ on $\Sigma$, the normalized entropy $\volen(g)^2\cdot\area(g)$ (which is invariant under scaling) satisfies
\begin{equation}\label{eqn_katok}
\volen(g)^2\cdot\area(g)\geq 2\pi|\chi(\Sigma)|=\volen(g_\hyp)^2\cdot\area(g_\hyp)
\end{equation}
where $g_\hyp$ is any hyperbolic metric on $\Sigma$ and $\chi(\Sigma)$ is the Euler characteristic of $\Sigma$. Furthermore, equality occurs if and only if $g$ is hyperbolic.
\end{itemize}

An important consequence of the second property is that
$$0<\volen(g)<+\infty.$$ This follows from the \v{S}varc-Milnor Lemma, which implies that the distance on $\widetilde{\Sigma}$ induced by a geodesic metric is isometric to the hyperbolic plane.

A convex projective structure on $\Sigma$ gives rise to the following objects on $\Sigma$ (see \cite{loftin_amer, labourie_cubic, benoist-hulin, nie_2} for details):
\begin{itemize}
\item a Finsler metric $g_\H$, called the \emph{Hilbert metric};
\item a Riemannian metric $g_\B$, called the \emph{Blaschke metric};
\item a holomorphic cubic differential $\vect{b}$ with respect to the conformal structure underlying $g_\B$, called the \emph{Pick differential}.
 \end{itemize} 
 Moreover, $g_\B$ and $\vect{b}$ satisfy \emph{Wang's equation}
\begin{equation}\label{eqn_wang}
\kappa_{g_\B}=-1+2||\vect{b}||^2_{g_\B}.
\end{equation}
Here $\kappa_{g_\B}$ is the curvature of $g_\B$ and $||\vect{b}||^2_{g_\B}$ is the pointwise norm of $\vect{b}$ with respect to $g_\B$, namely, if $g_\B=h(z)|\dz|^2$ and $\vect{b}=b(z)\dz^3$ in a local coordinate $z$, then 
$$
||\vect{b}||^2_{g_\B}(z):=\frac{|b(z)|^2}{h(z)^3}.
$$ 

The Labourie-Loftin theorem says that, for each pair $(\ac{J}, \vect{b})$ mentioned in the introduction, there is a unique Riemannian metric $g$ conformal to $\ac{J}$ such that $g$ and $\vect{b}$ are respectively the Blaschke metric and the Pick differential of a convex projective structure. This gives the identification $\mathcal{P}(\Sigma)\cong\mathcal{C}(\Sigma)$ discussed in  the introduction.

The entropy $\volen(g_\H)$ of the Hilbert metric is an interesting and systematically studied quantity because it equals the topological entropy of an interesting dynamical system -- the geodesic flow of a convex projective surface (see \cite{crampon}).

As in the introduction, with an abuse of notation, we also let $\volen: \mathcal{P}(\Sigma)\rightarrow \mathbb{R}_+$ denote the function assigning to each convex projective structure the entropy of its Hilbert metric.

\begin{proof}[Proof of the ``if'' part of the theorem]
Proposition 3.4 in \cite{benoist-hulin} implies that there is universal constant $c>1$ such that the Hilbert metric $g_\H$ and the Blaschke metric $g_\B$ of any convex projective structure on $\Sigma$ satisfy
$$
c^{-1}\|v\|_{g_\B}\leq \|v\|_{g_\H}\leq c\|v\|_{g_\B}
$$
for any tangent vector $v$ of $\Sigma$. Here $\|v\|_{g_\B}$ and $\|v\|_{g_\H}$ denotes the norm of $v$ with respect to $g_\B$ and $g_\H$, respectively. As a consequence, we have
\begin{equation}\label{eqn_bh}
c^{-1}\volen(g_\B)\leq\volen(g_\H)\leq c\,\volen(g_\B).
\end{equation}

Given a convex projective structure with Blaschke metric $g_\B$ and Pick differential $\vect{b}$, we consider the flat metric with conic singularity $|\vect{b}|^\frac{2}{3}$.
By definition, if $\vect{b}=b(z)\dz^3$ under a conformal local coordinate $z$, then $|\vect{b}|^\frac{2}{3}:=|b(z)|^\frac{2}{3}|\dz|^2$.

It is a well known fact that $g_\B$ has non-positive curvature (see \eg \cite{nie_2} Coro. 6.2), or equivalently, there is a pointwise majorization
\begin{equation}\label{eqn_bp}
g_\B\geq 2^\frac{1}{3}|\vect{b}|^\frac{2}{3}
\end{equation}
(the equivalence follows from Wang's equation (\ref{eqn_wang})). Therefore we have
\begin{equation}\label{eqn_ent}
\volen(g_\B)\leq \volen(2^\frac{1}{3}|\vect{b}|^\frac{2}{3}).
\end{equation}

Let $(\vect{b}_n)$ be a sequence in $V=H^0((\Sigma,\ac{J}),K^3)$ tending to infinity. In view of (\ref{eqn_bh}) and (\ref{eqn_ent}), in order to prove the ``if'' part of the theorem, it is sufficient to show that $\volen(2^\frac{1}{3}|\vect{b}_n|^\frac{2}{3})$ tends to $0$.

To this end, we note that the function 
$$
V\setminus\{0\}\rightarrow\mathbb{R}_{\geq 0},\quad \vect{b}\mapsto \volen(2^\frac{1}{3}|\vect{b}|^\frac{2}{3})
$$
is continuous, because for $\vect{b}, \vect{b}'\in V\setminus\{0\}$, the quasi-isometry constant $a$ in (\ref{eqn_quasiiso}) between the lifts of $|\vect{b}|^\frac{2}{3}$ and $|\vect{b}'|^\frac{2}{3}$ tends to $1$ as $\vect{b}'$ approaches $\vect{b}$.  Therefore, if we fix a norm $\|\cdot\|$ on $\V$ and let $S\subset\V$ be the unit sphere, then $M:=\max_{\vect{b}\in S}\volen(|\vect{b}|^\frac{2}{3})\in\mathbb{R}_+$ exists. Thus
\begin{align*}
\volen(2^\frac{1}{3}|\vect{b}_n|^\frac{2}{3})=\volen\left(2^\frac{1}{3}\|\vect{b}_n\|^\frac{2}{3}\left|\frac{\vect{b}_n}{\|\vect{b}_n\|}\right|^\frac{2}{3}\right)&=2^{-\frac{1}{3}}\|\vect{b}_n\|^{-\frac{2}{3}}\volen\left(\left|\frac{\vect{b}_n}{\|\vect{b}_n\|}\right|^\frac{2}{3}\right)\\
&\leq 2^{-\frac{1}{3}}\|\vect{b}_n\|^{-\frac{2}{3}}M
\end{align*}
The last term tends to $0$ because $\|\vect{b}_n\|$ tends to $+\infty$, as required.
\end{proof}

In order to prove the ``only if'' part of the theorem without using the continuity of the Labourie-Loftin bijection, as mentioned in the introduction, we need the following lemma.
\begin{lemma*}
Given a convex projective structure on $\Sigma$ with Blaschke metric $g_\B$ and Pick differential $\vect{b}$, the total area of $g_\B$ and that of the flat metric $2^\frac{1}{3}|\vect{b}|^\frac{2}{3}$ satisfy
$$
\area(g_\B)\leq 2\pi|\chi(\Sigma)|+\area(2^\frac{1}{3}|\vect{b}|^\frac{2}{3}).
$$
\end{lemma*}
\begin{proof}
The Wang's equation (\ref{eqn_wang}) satisfied by $g_\B$ and $\vect{b}$ can be written as
$$
1=-\kappa_{g_\B}+2\|\vect{b}\|^2_{g_\B}.
$$
Integrating both sides over $\Sigma$ with respect to the volume form of $g_\B$, then applying Gauss-Bonnet Formula to the first term on the right-hand side, we get
$$
\area(g_\B)=2\pi|\chi(\Sigma)|+\int_\Sigma2\|\vect{b}\|^2_{g_\B}\mathsf{dvol}_{g_\B}.
$$
To prove the lemma, it is sufficient to show that the integrand $2\|\vect{b}\|^2_{g_\B}\mathsf{dvol}_{g_\B}$ in the last term above is pointwise majorized by the volume form $\mathsf{dvol}'$ of $2^\frac{1}{3}|\vect{b}|^\frac{2}{3}$, but this follows from the inequality (\ref{eqn_bp}): assuming $g_\B=h(z)|\dz|^2$ and $\vect{b}=b(z)\dz^3$ in a local coordinate $z=x+\vect{i} y$, we have
$$
2\|\vect{b}\|^2_{g_\B}\mathsf{dvol}_{g_\B}=\frac{2|b|^2}{h^2}\dx\wedge\dif y\leq \frac{2|b|^2}{(2^\frac{1}{3}|b|^\frac{2}{3})^2}\dx\wedge\dif y=2^\frac{1}{3}|b|^\frac{2}{3}\dx\wedge\dif y=\mathsf{dvol}'.
$$
\end{proof}

\begin{proof}[Proof of the ``only if'' part of the theorem]
Let $g_\B^{(n)}$ be the Blaschke metric of the convex projective structure corresponding to $(\ac{J},\vect{b}_n)$ and assume that
$$
\lim_{n\rightarrow+\infty}\volen(g_\B^{(n)})=0.
$$
We need to prove that the cubic differential $\vect{b}_n$ tends to infinity in $H^0((\Sigma,\ac{J}),K^3)$, or equivalently, the total area of the associated flat metric
$|\vect{b}_n|^\frac{2}{3}$ tends to infinity. But this follows from Katok's inequality (\ref{eqn_katok}) and the above lemma, because they imply
$$
\frac{2\pi|\chi(\Sigma)|}{\volen(g_\B^{(n)})^2}\leq \area(g_\B^{(n)})\leq 2\pi|\chi(\Sigma)|+\area(2^\frac{1}{3}|\vect{b}_n|^\frac{2}{3}).
$$
This completes the proof of the theorem.
\end{proof}

To conclude, we give some comments on the case where $\Sigma$ is a punctured surface of finite type, \ie  $\Sigma$ is obtained from a closed oriented surface by removing finitely many punctures. As a generalization of the Labourie-Loftin identification $\mathcal{P}(\Sigma)\cong\mathcal{C}(\Sigma)$, Benoist and Hulin \cite{benoist-hulin} identified the space of convex projective structures with finite Hilbert volume with the space of those pairs $(\ac{J},\vect{b})$ where $\ac{J}$ is a punctured Riemann surface structure and $\vect{b}$ is a holomorphic cubic differential with at most second order pole at each puncture. We gave in \cite{nie_2} a further generalization. 

The above theorem holds perhaps in these more general settings as well. However, the above proof does not work. Indeed, at a pole of order $\leq 2$ , the metric $|\vect{b}|^\frac{2}{3}$ is incomplete and its pullback to the universal cover is not quasi-isometric to the hyperbolic plane, hence whether $|\vect{b}|^\frac{2}{3}$ has finite entropy remains a problem.  

\subsection*{Acknowledgement} We would like to thank the referee for a helpful remark on the first draft of the paper.

\bibliographystyle{amsalpha} \bibliography{entropysurface}

\providecommand{\bysame}{\leavevmode\hbox to3em{\hrulefill}\thinspace}
\providecommand{\MR}{\relax\ifhmode\unskip\space\fi MR }
\providecommand{\MRhref}[2]{%
  \href{http://www.ams.org/mathscinet-getitem?mr=#1}{#2}
}
\providecommand{\href}[2]{#2}
\begin{thebibliography}{{Zha}13}

\bibitem[BH13]{benoist-hulin}
Yves Benoist and Dominique Hulin, \emph{Cubic differentials and finite volume
  convex projective surfaces}, Geom. Topol. \textbf{17} (2013), no.~1,
  595--620. \MR{3039771}

\bibitem[Cra09]{crampon}
Micka{\"e}l Crampon, \emph{Entropies of strictly convex projective manifolds},
  J. Mod. Dyn. \textbf{3} (2009), no.~4, 511--547. \MR{2587084 (2011g:37079)}

\bibitem[Kat88]{katok}
Anatole Katok, \emph{Four applications of conformal equivalence to geometry and
  dynamics}, Ergodic Theory Dynam. Systems \textbf{8$\sp *$} (1988),
  no.~Charles Conley Memorial Issue, 139--152. \MR{967635 (89m:58165)}

\bibitem[Lab07]{labourie_cubic}
Fran{\c{c}}ois Labourie, \emph{Flat projective structures on surfaces and cubic
  holomorphic differentials}, Pure Appl. Math. Q. \textbf{3} (2007), no.~4,
  part 1, 1057--1099. \MR{2402597 (2009c:53046)}

\bibitem[Lof01]{loftin_amer}
John~C. Loftin, \emph{Affine spheres and convex {$\Bbb{RP}\sp n$}-manifolds},
  Amer. J. Math. \textbf{123} (2001), no.~2, 255--274. \MR{1828223
  (2002c:53018)}

\bibitem[Lof07]{loftin_limit}
John Loftin, \emph{Flat metrics, cubic differentials and limits of projective
  holonomies}, Geom. Dedicata \textbf{128} (2007), 97--106. \MR{2350148
  (2009c:53011)}

\bibitem[Nie]{nie_1}
Xin Nie, \emph{{On the Hilbert geometry of simplicial Tits sets}},
  {\href{http://arxiv.org/abs/0801.2789}{\texttt{arXiv:0902.0885}}}, \emph{to
  appear in} Ann. Inst. Fourier (2015).

\bibitem[{Nie}15]{nie_2}
X.~{Nie}, \emph{{Meromorphic cubic differentials and convex projective
  structures}}, ArXiv:1503.02608 (2015).

\bibitem[Par12]{parreau}
Anne Parreau, \emph{Compactification d'espaces de repr\'esentations de groupes
  de type fini}, Math. Z. \textbf{272} (2012), no.~1-2, 51--86. \MR{2968214}

\bibitem[{Zha}13]{zhang}
T.~{Zhang}, \emph{{The degeneration of convex RP\^{}2 structures on surfaces}},
  ArXiv:1312.2452 (2013).

\end{thebibliography}

\end{document}